\documentclass{amsart}   	
\usepackage{lineno,hyperref,verbatim, xcolor}
\modulolinenumbers[5]

\usepackage{geometry}                		
\usepackage{amsmath}
\usepackage{enumerate}
\usepackage{amsthm}	
\usepackage{amssymb}
\usepackage{bbold}
\usepackage{tikz}
\usepackage[toc]{appendix}
\newtheorem{thm}{Theorem}[section]

\newtheorem{rem}[thm]{Remark}
\newtheorem{prop}[thm]{Proposition}

\newtheorem{lem}[thm]{Lemma}

\newtheorem*{claim}{Claim}

\newtheorem*{thm*}{Theorem}
\newtheorem*{question*}{Question}
\newtheorem*{assumption*}{Assumption}

\theoremstyle{definition}
\newtheorem{example}[thm]{Example}

\newcommand{\F}{\mathbb{F}}
\newcommand{\fl}{\mathfrak{l}}
\newcommand{\fp}{\mathfrak{p}}
\newcommand{\Gal}{\mathrm{Gal}}
\newcommand{\Aut}{\mathrm{Aut}}
\newcommand{\sep}{\mathrm{sep}}
\newcommand{\alg}{\mathrm{alg}}
\newcommand{\GL}{\mathrm{GL}}
\renewcommand{\ss}{\mathrm{ss}}

\title{Galois criterion for torsion points of Drinfeld modules}
\author{Chien-Hua Chen}

\begin{document}
\begin{abstract}
In this paper, we formulate the Drinfeld module analogue of a question raised by Lang and studied by Katz on the existence of rational points 
on abelian varieties over number fields. Given a maximal ideal $\fl$ of $\F_q[T]$, the question 
essentially asks whether, up to isogeny, a Drinfeld module $\phi$ over $\F_q(T)$ contains a rational $\fl$-torsion point if the reduction of $\phi$ at almost all primes of $\F_q[T]$ contains a rational $\fl$-torsion point. Similar to the case of abelian varieties, we show that 
the answer is positive if the rank of the Drinfeld module is $2$, but negative if the rank is $3$. Moreover, for rank $3$ Drinfeld modules 
we classify those cases where the answer is positive. 
\end{abstract}
\maketitle

\section{Introduction}

Let $l$ be a rational prime and $A$ an abelian variety over a numder field $K$. Let $S$ be a set of primes of $K$ of density $1$ where 
$A$ has good reduction. For $\fp\in S$ denote by $A_\fp$ the reduction of $A$ at $\fp$ and by $\F_\fp$ the residue field at $\fp$. 
In \cite{Kat81}, N. Katz proved the following: if $A_\fp(\F_\fp)[l]\neq 0$ for all $\fp\in S$, then there exists an abelian variety $A'$ isogenous 
to $A$ over $K$ such that $A'(K)[l]\neq 0$, provided $\dim A'\leq 2$. He also constructed counterexamples in all dimensions $\geq 3$.   
This answered a question posed by Serge Lang. 

The analogy between Drinfeld modules and abelian varieties is well-known and has been extensively studied. In this paper we consider the 
analogue of Lang's question for Drinfeld modules. Before we state our main results, we note that Lang's question can be reformulated 
as follows. Let 
$$\bar{\rho}_l:\Gal(\bar{K}/K)\to \Aut(A[l])
$$ 
be the representation arising from the action of the absolute Galois group of $K$ on $A[l]$. If for every $\sigma\in \Gal(\bar{K}/K)$  
we have $\det(1-\bar{\rho}_l(\sigma))=0$, is it true that the semi-simplification of $A[l]$ contains the trivial representation? 
In fact, Katz approaches Lang's question from this group-theoretic perspective.

Now let $A=\F_q[T]$ be the ring of polynomials in indeterminate $T$ with coefficients in $\F_q$, 
where $q$ is a power of a prime $p\geq 5$. Let $F=\F_q(T)$ be the 
field of fractions of $A$. Let $\phi$ be a Drinfeld $A$-module of rank $r$ over $F$. Let $S$ be a set of maximal ideals of $A$ of density $1$ where 
$\phi$ has good reduction. Let $\F_\fp=A/\fp$ denote the residue field at $\fp$ and $\phi\otimes \F_\fp$ the reduction of $\phi$ 
at $\fp\in S$. Given a maximal ideal $\fl$ of $A$, 
the absolute Galois group $G_F:=\Gal(F^\sep/F)$ of $F$ acts on $\phi[\fl]$. Let 
$$
\bar{\rho}_{\phi, \fl}: G_F\to \Aut(\phi[\fl])\cong \GL_r(\F_\fl)
$$  
be the corresponding representation. 

The main results of this paper are the following:

\begin{thm}\label{mainthm1} Assume $r=2$. Let $\fl$ be a maximal ideal of $A$. 
 Suppose for all $\fp\in S$ the reduction $\phi\otimes \F_\fp$ has a nontrivial $\fl$-torsion point defined over $\F_\fp$. 
 Then the semisimplification of the $G_F$-module $\phi[\fl]$ contains the trivial representation. 
\end{thm}

\begin{thm}\label{mainthm2} Assume $r=3$. 
	Let $\fl$ be a maximal ideal of $A$. 
	Suppose for all $\fp\in S$ the reduction $\phi\otimes \F_\fp$ has a nontrivial $\fl$-torsion point defined over $\F_\fp$. 
\begin{itemize}
 \item[(1)] If $\det(\bar{\rho}_{\phi, \fl})$ is a nontrivial representation, then the semisimplification of $\phi[\fl]$ contains the trivial representation.
 \item[(2)] If $\det(\bar{\rho}_{\phi, \fl})$ is the trivial representation and $\phi[\fl]$ is reducible, then either the semisimplification of $\phi[\fl]$ contains the  trivial representation or there is a quadratic twist $\phi'$ of $\phi$ such that the semisimplification of $\phi'[\fl]$ contains the trivial representation. Moreover, there do exist Drinfeld modules $\phi$ for which 
 $\det(\bar{\rho}_{\phi, \fl})$ is the trivial representation, $\phi[\fl]$ is reducible, but 
the semisimplification of $\phi[\fl]$ does not contain the trivial representation. 
 \item[(3)] There are Drinfeld modules $\phi$ which satisfy the assumptions of the theorem but for which $\phi[\mathfrak{l}]$ is irreducible, so 
 its semisimplification does not contain the trivial representation.
\end{itemize}
\end{thm}

In particular, the analogue of Lang's question has a positive answer for Drinfeld modules of rank $2$, 
but can have a negative answer for Drinfeld modules of rank $3$. 


To prove the main theorems, we analyze the structure of the Galois module $\phi[\mathfrak{l}]$ following the 
strategy  in Katz's paper \cite{Kat81}. 
A crucial idea  
in our proof of part (3) of Theorem \ref{mainthm2} was inspired by Cullinan's paper \cite{Cull19}. 

To conclude the introduction, we observe the following relation of our results to the existence of rational points, up to isogeny. 
Suppose the semisimplification of $\phi[\mathfrak{l}]$ contains the trivial representation. If the trivial representation is 
is a $G_F$-submodule of $\phi[\mathfrak{l}]$ then $\phi$ has an $F$-rational $\fl$-torsion point. Otherwise, 
there is a $G_F$-submodule $H\subset \phi[\mathfrak{l}]$ such that the trivial representation is a $G_F$-submodule 
of $H':=\phi[\mathfrak{l}]/H$. In this case, 
we consider the $F$-isogenous Drinfeld module $\phi':=\phi/H$ of $\phi$ 
(see \cite{Gos96} Proposition 4.7.11 for the construction of $\phi'$). 
Then we have $H'\subset \phi'[\mathfrak{l}]$, which implies that $\phi'[\mathfrak{l}]$ has a nontrivial $F$-rational point.

\section{Preliminaries}

An {\textbf{$A$-field}} is a field $K$ equipped with a homomorphism $\gamma: A \rightarrow K$ of $\mathbb{F}_q$-algebras. 
The kernel $\ker(\gamma)$ is called the {\textbf{$A$-characteristic}} of $K$; 
we say $K$ has {\textbf{generic characteristic}} if $\ker(\gamma)=0$.

Let $K\{\tau\}$ be the ring of skew polynomials satisfying the commutation rule $\tau \cdot c = c^q\cdot \tau$.
A \textbf{Drinfeld $A$-module over $K$ of rank $r\geqslant 1$} is a ring homomorphism
\begin{align*}
\phi&: A\longrightarrow K\{\tau\} \\
      &\ \ \ a \longmapsto \phi_a=\gamma(a)+ \sum^{r\cdot {\rm{deg}}(a)}_{i=1}g_i(a)\tau^i.
\end{align*}
It is uniquely determined by $\phi_T=\gamma(T)+\sum^{r}_{i=1}g_i(T)\tau^i$, where $g_r(T)\neq0$. 
An \textbf{isogeny} from a Drinfeld module $\phi$ to another Drinfeld module $\psi$ over $K$ is a nonzero element $u\in K\{\tau\}$ such that $u\cdot \phi_a=\psi_a \cdot u$ for all $a\in A$. 

The Drinfeld module $\phi$ over $K$ gives $K$ an $A$-module structure, where $a\in A$ acts on $K$ via $\phi_{a}$. We use the notation 
$^{\phi}K$ to emphasize the action of $A$ on $K$.
The {\textbf{$a$-torsion}} is 
$$\phi[a]=\left\{ \alpha\in K^\alg\mid \phi_a(\alpha)= \gamma(a)\alpha+ \sum^{r\cdot {\rm{deg}}(a)}_{i=1}g_i(a)\alpha^{q^i}=0 \right\}.$$ 
Note that $\phi[a]$ is an $A$-module, where $b\in A$ acts on $\alpha\in \phi[a]$ by 
$b\cdot \alpha=\phi_b(\alpha)$. The following is well-known and easy to prove (cf. \cite{Gos96}): 

\begin{prop}\label{prop0.2}
Let $\phi$ be a Drinfeld module over $K$ of rank $r$ and $0\neq a\in A$. If the $A$-characteristic of $K$ does not divide $a$, then 
there is an isomorphism of $A$-modules $\phi[a]\simeq (A/aA)^r$.
\end{prop}

Note that if the characteristic of $K$ does not divide $a$, then $\phi_a(x)=\gamma(a)x+\sum^{r\cdot {\rm{deg}}(a)}_{i=1}g_i(a)x^{q^i}$ is a separable polynomial, so $G_K=\Gal(K^\sep/K)$ acts on $\phi[a]$ and this action commutes with the action of $A$. From this action we get a 
representation $G_K\to \Aut_A(\phi[a])\cong \GL_r(A/aA)$. When $a=\fl$ is irreducible, we denoted this representation $\bar{\rho}_{\phi, \fl}$. 
Taking inverse limit with respect to $\fl^i$, we get the \textbf{$\fl$-adic Galois representation}
$${\rho}_{\phi,\fl}: G_K \longrightarrow \varprojlim_{i}{\Aut}(\phi[\fl^i])\cong \GL_r(A_\fl),$$
where $A_\fl$ denotes the completion of $A$ at $\fl$. 

Let $\phi$ be a Drinfeld module over $F$ defined by $$\phi_T=T+g_1\tau+g_2\tau^2+\cdots+g_r\tau^r.$$ 
We say that $\phi$ has \textbf{good reduction} at the maximal ideal $\fp$ of $A$ if all $g_i$ are integral at $\fp$, i.e., lie in $A_\fp$, 
and $g_r$ is a unit in $A_\fp$. The \textbf{reduction}  of $\phi$ at $\fp$ 
is the Drinfeld module $\phi\otimes \F_\fp$ over $\F_\fp$ defined by 
$(\phi\otimes \F_\fp)_T=\bar{T}+\bar{g_1}\tau+\cdots+\bar{g_r}\tau^r$, where $\bar{g_i}$ is the reduction of $g_i$ modulo $\fp$. 

If $\mathfrak{p} \neq \mathfrak{l}$ is a prime of good reduction of $\phi$, then the $\mathfrak{l}$-adic Galois representation ${\rho}_{\phi,\mathfrak{l}}$ is unramified at $\mathfrak{p}$. Therefore, the matrix ${\rho}_{\phi,\mathfrak{l}}({\rm{Frob}}_{\mathfrak{p}}) \in {\rm{GL}}_r(A_\mathfrak{l})$ is well-defined up to conjugation, so we can consider the characteristic polynomial $P_{\phi,\mathfrak{p}}(x)=\det(xI-{\rho}_{\phi,\mathfrak{l}}({\rm{Frob}}_{\mathfrak{p}}))$ of the Frobenius element ${\rm{Frob}_{\mathfrak{p}}}$ 
at $\fp$. 
It is known that the coefficients of polynomial $P_{\phi,\mathfrak{p}}(x)$ are independent of the choice of $\mathfrak{l}$ and belong to $A$ (see \cite{Gek91} Corollary 3.4). 
Moreover, $P_{\phi,\mathfrak{p}}(x)$ is equal to the characteristic polynomial of Frobenius endomorphism of $\phi\otimes\mathbb{F}_{\mathfrak{p}}$ acting on $T_{\mathfrak{l}}(\phi\otimes\mathbb{F}_{\mathfrak{p}})$.

We will need a fact about the value $P_{\phi,\mathfrak{p}}(1)$ which is the analogue of Hasse's theorem about the number of 
rational points on an elliptic curve over a finite field. 
By the structure theorem for finitely generated modules 
over principal ideal domains, we have an isomorphism of $A$-module 
$$^{\phi\otimes\mathbb{F}_\mathfrak{p}}\mathbb{F}_\mathfrak{p}\cong A/b_1A\times\cdots\times A/b_sA,$$ 
for uniquely determined monic polynomials $b_1\mid b_2\mid \cdots \mid b_s$. 

\begin{prop}\label{red}
We have an equality of ideals 
$\left(\prod_{i=1}^{s}b_i\right)= (P_{\phi,\mathfrak{p}}(1))$.
\end{prop}
\begin{proof}
See \cite{Gek91}.
\end{proof}

We conclude this section by recalling the Brauer-Nesbitt Theorem. 
Let $G$ be a finite group and $V$ be a representation of $G$ defined over a field $K$ of characteristic $p$. 
The {\textbf{semisimplification}} $V^{\ss}$ of $V$ is the direct sum of Jordan-H\"older constituents of the $K[G]$-module $V$. In other words, suppose the Jordan-H\"older series of the $K[G]$-module $V$ is $V=V_0\supset V_1\supset V_2\supset\cdots\supset V_n=\{0\}$, then $$V^{ss}=\oplus_{i=0}^{n-1}V_i/V_{i+1}.$$

\begin{thm}[Brauer-Nesbitt Theorem]\label{bnt}
Let $G$ be a finite group. Let $V$ and $W$ be two $K[G]$-modules which are finite dimensional as $K$-vector spaces. If for all $g\in G$, the characteristic polynomial of $g$ acting on $V$ and $W$ are equal, then $V$ and $W$ have the same Jordan-H\"older constituents. In other words, the semisimplification $V^\ss$ is isomorphic to $W^\ss$ as $K[G]$-modules.
\end{thm}
\begin{proof}
See \cite{CuRe62}, p.215.
\end{proof}

\section{Proof of Theorem \ref{mainthm1}}

Let $\phi$ be a Drinfeld $A$-module over $F$ of rank $2$ and $\mathfrak{l}$ be a prime ideal of $A$. Replacing $\phi$ by $c\phi c^{-1}$ with suitable $c\in A$ such that $(c\phi c^{-1})_T\in A\{\tau\}$, we may assume $\phi$ is defined over $A$. 
Let $S$ be a subset of primes of $A$ with density $1$ where $\phi$ has good reduction.   
Assume:
\begin{center}
	For all $\fp\in S$, the space $(\phi\otimes\mathbb{F}_\mathfrak{p})[\mathfrak{l}]$ contains a nontrivial $\mathbb{F}_\mathfrak{p}$-rational point.
\end{center}
The assumption is equivalent to saying that the $A$-module $^{\phi\otimes\mathbb{F}_\mathfrak{p}}\mathbb{F}_\mathfrak{p}$ has a nontrivial $\mathfrak{l}$-torsion. Hence Proposition \ref{red} implies $P_{\phi,\mathfrak{p}}(1)$ is divisible by $\mathfrak{l}$ for almost all prime ideal $\mathfrak{p}$ of $A$. Thus we have
$$P_{\phi,\mathfrak{p}}(1)=1-{\rm{tr}}({\rm{Frob}}_\mathfrak{p})+\det({\rm{Frob}}_\mathfrak{p})\equiv 0 \mod \mathfrak{l}.$$
The Chebotarev density theorem implies that for all $g\in G_F$, the action of $g$ on $\phi[\mathfrak{l}]$ satisfies
$${\rm{tr}}(g)\equiv 1+\det(g) \mod \mathfrak{l}.$$
Now we can compare the $G_F$-modules $\phi[\mathfrak{l}]$ and $\mathbb{1}\oplus\det(\phi[\mathfrak{l}])$, where $\mathbb{1}$ denotes the trivial representation. For any $g\in G_F$, we know the characteristic polynomial of $g$ acting on $\phi[\mathfrak{l}]$ and $\mathbb{1}\oplus\det(\phi[\mathfrak{l}])$ are the same. Hence Theorem \ref{bnt} implies $\phi[\mathfrak{l}]$ and $\mathbb{1}\oplus\det(\phi[\mathfrak{l}])$ are isomorphic up to semisimplification. 

However, Theorem \ref{mainthm1} does not imply that $\phi[\mathfrak{l}]$ contains a nontrivial $\mathfrak{l}$-torsion point defined over $F$. We give an example showing that there is a Drinfeld $A$-module $\phi$ of rank $2$ 
such that $\phi[\fl]^\ss$ contains the trivial representation while $\phi[\mathfrak{l}]$ has no nontrivial points defined over $F$. 

\begin{example}
	Consider the Drinfeld module $\phi$ over $F$ defined by
	$$\phi_T(x)=\prod^{q-1}_{i=0}(x^q+Tx-i).$$
	Let $e_1$ be a nonzero root of $x^q+Tx$ and $e_2$ be a nonzero root of $x^q+Tx-1$. Then, $\phi[T]$ as an $\F_q$-vector space has 
	a a basis $\{e_1, e_2\}$. An element $g\in G_F$ act on this basis by 
	\begin{align*}
	ge_1&= c_ge_1, \quad c_g\in \mathbb{F}_q^*,\\
	ge_2&=e_2+d_ge_1,\quad d_g\in \mathbb{F}_q. 
	\end{align*}
Hence $\bar{\rho}_{\phi, T}$ is the Galois representation
$$\begin{array}{rccl}
&G_F&\rightarrow&{\rm{GL}}(V)\cong {\rm{GL}}_2(\mathbb{F}_q)\\
&g&\mapsto&\left(\begin{array}{cc}c_g & d_g \\0 & 1\end{array}\right)
\end{array}.
$$
Since there are $c_g\neq 1$, this representation does not contain the trivial representation as a direct summand. 
\end{example}

\section{Proof of Theorem \ref{mainthm2}}

\subsection{Basic setting}
Similar to the previous section, let $\phi$ be a Drinfeld $A$-module over $F$ of rank $3$ and $\mathfrak{l}$ be a prime ideal of $A$. 
Let $S$ be a subset of primes of $A$ with density $1$ where $\phi$ has good reduction.  Assume:
\begin{center}
	For all $\fp\in S$, the space $(\phi\otimes\mathbb{F}_\mathfrak{p})[\mathfrak{l}]$ contains a nontrivial $\mathbb{F}_\mathfrak{p}$-rational point.
\end{center}

Again, using Proposition \ref{red}, we get 
$$P_{\phi,\mathfrak{p}}(1)=\det(1-{\rm{Frob}}_\mathfrak{p})\equiv 0 \mod \mathfrak{l}$$
for all $\mathfrak{p}\in S$. Then, by the Chebotarev density theorem, the characteristic polynomial of any $g\in G_F$ acting on $\phi[\mathfrak{l}]$ satisfies
$$\det(1-g)=\sum_{i=0}^{3}(-1)^i{\rm{Tr}}(g|_{\Lambda^i(\phi[\mathfrak{l}])})=0.$$
Here $\Lambda^i(\phi[\mathfrak{l}])$ are exterior powers of $\phi[\mathfrak{l}]$, and $\Lambda^0(\phi[\mathfrak{l}])$ is defined to be the trivial representation $\mathbb{1}$. Therefore, we have 
$${\rm{Tr}}(\mathbb{1}\oplus \Lambda^2(\phi[\mathfrak{l}]))={\rm{Tr}}(\phi[\mathfrak{l}]\oplus \det(\phi[\mathfrak{l}])).$$
Now we can compare the $\mathbb{F}_\mathfrak{l}[G_F]$-modules $\mathbb{1}\oplus \Lambda^2(\phi[\mathfrak{l}])$ and $\phi[\mathfrak{l}]\oplus \det(\phi[\mathfrak{l}])$. For any $g\in G_F$, they both have the same trace. But a stronger claim is true: 
\begin{claim}
$g$ on both spaces has the same characteristic polynomial.
\end{claim}
\begin{proof}[Proof of claim]
Let $\lambda_1, \lambda_2, \lambda_3$ be the characteristic values of $g$ acting on $\phi[\mathfrak{l}]$. Then the characteristic values of $g$ acting on $\Lambda^2(\phi[\mathfrak{l}])$ and $\det(\phi[\mathfrak{l}])$ are $\lambda_1\lambda_2,\ \lambda_1\lambda_3,\ \lambda_2\lambda_3$ and $\lambda_1\lambda_2\lambda_3$, respectively. The action of $g$ on $\mathbb{1}\oplus \Lambda^2(\phi[\mathfrak{l}])$ and $\phi[\mathfrak{l}]\oplus \det(\phi[\mathfrak{l}])$ have the same trace
implies

$$
\begin{array}{rccl}

&\lambda_1+\lambda_2+\lambda_3+\lambda_1\lambda_2\lambda_3&=&1+\lambda_1\lambda_2+\lambda_1\lambda_3+\lambda_2\lambda_3.\\

\Rightarrow&\lambda_1+\lambda_2-\lambda_1\lambda_2-1&=&(\lambda_1+\lambda_2-\lambda_1\lambda_2-1)\lambda_3\\

\Rightarrow&(\lambda_1-1)(1-\lambda_2)&=&(\lambda_1-1)(1-\lambda_2)\lambda_3
\end{array}
$$
Hence one of $\lambda_1$, $\lambda_2$, and $\lambda_3$ must be equal to $1$, which implies the characteristic values of $g$ acting on both spaces are the same.
\end{proof}
As we have deduced that the $G_F$-action on $\mathbb{1}\oplus \Lambda^2(\phi[\mathfrak{l}])$ and $\phi[\mathfrak{l}]\oplus \det(\phi[\mathfrak{l}])$ have the same characteristic polynomials, the Brauer-Nesbitt theorem implies both representations are isomorphic up to semisimplification. In other words, we have

\begin{equation}\label{eqMihranInserted}
\mathbb{1}\oplus \Lambda^2(\phi[\mathfrak{l}])^{\ss}\cong \phi[\mathfrak{l}]^{\ss}\oplus \det(\phi[\mathfrak{l}]).
\end{equation}

\subsection{Case (1)} Suppose 
$\det(\phi[\mathfrak{l}])\neq \mathbb{1}$. Then, the  isomorphism \eqref{eqMihranInserted} implies that 
the semisimplification of $\phi[\mathfrak{l}]$ contains the  trivial representation. 


\subsection{Case (2)}
Now assume $\det(\phi[\mathfrak{l}])=\mathbb{1}$, the $G_F$-module $\phi[\mathfrak{l}]$ is reducible, and the semisimplification of $\phi[\mathfrak{l}]$ does not contain the trivial representation. 

We have the isomorphisms of $G_F$-modules
$$\Lambda^2(\phi[\mathfrak{l}])\cong{\rm{Hom}}(\phi[\mathfrak{l}],\det(\phi[\mathfrak{l}]))\cong\phi[\mathfrak{l}]^{\lor}\otimes\det(\phi[\mathfrak{l}])\cong\phi[\mathfrak{l}]^{\lor}, $$
where the first isomorphism is given by 
$$
\begin{array}{clc}
\Lambda^2(\phi[\mathfrak{l}])&\rightarrow&{\rm{Hom}}(\phi[\mathfrak{l}],\det(\phi[\mathfrak{l}]))\\
u\land v&\mapsto&(w\mapsto u\land v\land w)

\end{array}
$$
Combining this with \eqref{eqMihranInserted}, we get
$$\phi[\mathfrak{l}]^{ss}\cong(\phi[\mathfrak{l}]^{\lor})^{ss},$$
i.e. the semisimplification of $\phi[\mathfrak{l}]$ is self dual.

Now we can consider the Jordan-H\"older series of $\phi[\mathfrak{l}]$.
\begin{enumerate}
\item[Case (i).] $\phi[\mathfrak{l}]\supset V\supset \{0\}$, where $V$ is an irreducible $G_F$-submodule of $\phi[\mathfrak{l}]$ with dimension $1$ or $2$.

In this case, the action of $g\in G_F$ on $\phi[\mathfrak{l}]^{ss}$ with respect to some basis is of the form  

$$\left(\begin{array}{c|c}\chi(g) &  \\\hline  & g|_{V}\end{array}\right)\ {\rm{or}}\ \left(\begin{array}{c|c}g|_{\phi[\mathfrak{l}]/V} &  \\\hline  & \chi(g)\end{array}\right).$$

Here $\chi:G_F\rightarrow \mathbb{F}_\mathfrak{l}^*$ is a nontrivial character of $G_F$. Since we have  
$\phi[\mathfrak{l}]^{ss}\cong(\phi[\mathfrak{l}]^{\lor})^{ss}$, then the character $\chi$ always satisfies
$$\chi(g)=\chi(g^{-1})\quad \forall\ g\in G_F.$$
Thus $\chi:G_F\rightarrow \{\pm 1\}$ is a nontrivial quadratic character. We may find some non-square element $c\in F$ such that $\chi$ can be rewritten in this way:
$$
\begin{array}{cccl}
\chi:&{\rm{Gal}}(F(\sqrt{c})/F)&\rightarrow&\{\pm1\} \\
       &g&\mapsto&\chi(g)
\end{array}
.$$
where $g\cdot\sqrt{c}=\chi(g)\cdot\sqrt{c}$.

To construct a suitable quadratic twist, we need the following Lemma: 

\begin{lem}\label{twist}
If the quadratic twist $\phi'$ of $\phi$ is given by 
$$\phi_T'=\sqrt{c}\cdot\phi_T\cdot\sqrt{c}^{-1},$$
then 
$${\bar{\rho}}_{\phi',\mathfrak{l}}\cong {\bar{\rho}}_{\phi,\mathfrak{l}}\otimes\chi$$
\end{lem}
\begin{proof}
Apply Lemma 4.1 in \cite{ChLe19} 
\end{proof}

With the Lemma above, we know that the action of $g\in G_F$ on $\phi'[\mathfrak{l}]^{ss}$ with respect to some basis is of the form  

$$\left(\begin{array}{c|c}\chi^2(g) &  \\\hline  & *\end{array}\right)\ {\rm{or}}\ \left(\begin{array}{c|c}* &  \\\hline  & \chi^2(g)\end{array}\right).$$
Combining with the fact that $\chi$ is a quadratic character, we can deduce that $\phi'[\mathfrak{l}]^{ss}$ contains the trivial representation.

\item[Case (ii).] $\phi[\mathfrak{l}]\supset V_1\supset V_2\supset\{0\}$, where $V_1$ and $V_2$ are $G_F$-submodules of $\phi[\mathfrak{l}]$ with dimension $2$ and $1$, respectively.
In this case, the action of $g\in G_F$ on $\phi[\mathfrak{l}]^{ss}$ with respect to some basis is of the form
$$\left(\begin{array}{ccc}\chi_1(g) &  &  \\ & \chi_2(g) &  \\ &  & \chi_3(g)\end{array}\right)$$
Here $\chi_i:G_F\rightarrow \mathbb{F}_\mathfrak{l}^*$ are nontrivial characters of $G_F$ for $i=1,2,3$. The self-duality of $\phi[\mathfrak{l}]$ then implies that for each $i\in\{1,2,3\}$,
$$\chi_i(g)=\chi_j(g^{-1})\ {\rm{for\ some\ }}j\in\{1,2,3\}.$$
If $i\neq j$, then we may assume that $\chi_1(g)=\chi_2(g^{-1})$ for all $g\in G_F$. The assumption that $\det(\phi[\mathfrak{l}])$ is the trivial representation implies $\chi_3(g)=1$ for all $g\in G_F$. Hence $\phi[\mathfrak{l}]^{ss}$ contains a trivial representation, contradicts to our assumption. Therefore, we have
$$\chi_i(g)=\chi_i(g^{-1})\ {\rm{for\ all\ }}i\in\{1,2,3\}.$$
In other words, each $\chi_i$ is a nontrivial quadratic character of $G_F$. Now we can focus on some $\chi_i$ and repeat what we did in case 1. There is some non-square element $c\in F$ such that
$$
\begin{array}{cccl}
\chi_i:&{\rm{Gal}}(F(\sqrt{c})/F)&\rightarrow&\{\pm1\} \\
       &g&\mapsto&\chi_i(g)
\end{array}
.$$
where $g\cdot\sqrt{c}=\chi_i(g)\cdot\sqrt{c}$.

Lemma \ref{twist} then implies the quadratic twist $\phi_T'=\sqrt{c}\cdot\phi_T\cdot\sqrt{c}^{-1}$ satisfies the property that $\phi'[\mathfrak{l}]^{ss}$ contains a trivial representation.
\end{enumerate}
Therefore, the second case of Theorem \ref{mainthm2} has been proved.

\begin{example}
In this example, we prove the existence of a Drinfeld module $\phi$ in Case (2) with the semisimplification of mod $(T)$ representation $\phi[T]$ does not contain the trivial representation. 
Let $A=\mathbb{F}_q[T]$ with $q=p^e$ be a prime power with $p\geqslant 5$, and $F=\mathbb{F}_q(T)$. Let $c_1$ and $c_2$ be two distinct non-square elements in $F$. Consider the $\mathbb{F}_q$-vector space
$$V=\sqrt{c_1}\cdot\mathbb{F}_q+\sqrt{c_2}\cdot\mathbb{F}_q+\sqrt{c_1c_2}\cdot\mathbb{F}_q.$$

The action of Galois group $G_F$ on $V$ gives a $G_F$-module structure. With respect to the basis $\{\sqrt{c_1}, \sqrt{c_2}, \sqrt{c_1c_2}\}$, we can see that $g\in G_F$ acts on $V$ via the matrix

$$\left(\begin{array}{ccc}\chi_1(g) & 0 & 0 \\0 & \chi_2(g) & 0 \\0 & 0 & \chi_1(g)\chi_2(g)\end{array}\right).$$
Here $\chi_1$ and $\chi_2$ are quadratic characters. 

By the Boston-Ose theorem (\cite{Bose00} Theorem 6.1), the representation $V$ is arised from the mod $(T)$ Galois representation $\bar{\rho}_{\phi,T}$ associated to some Drinfeld $A$-module $\phi$ over $F$ of rank $3$. Therefore, there is some basis of $\phi[T]$ such that $g\in G_F$ acts on $\phi[T]$ via the above matrix with respect to the basis. Since $\chi_1$ and $\chi_2$ are quadratic characters, one of $\chi_1(g)$, $\chi_2(g)$ and $\chi_1(g)\chi_2(g)$ must equal to $1$ for each $g$. Thus the characteristic polynomial of each $g\in G_F$ acting on $\phi[T]$ must has a factor $x-1$ and so the Drinfeld module satisfies the assumption in Theorem \ref{mainthm2}. The reason why the semisimplification $\phi[T]^{ss}$ does not contain the trivial representation is because $\chi_1$ and $\chi_2$ are distinct quadratic characters. Thus for each one of $\chi_1$, $\chi_2$, and $\chi_1\chi_2$, there is some Galois element $g\in G_F$ that makes it not equal to $1$.

\end{example}

\subsection{Case (3)} 
In this section we construct a Drinfeld module $\phi$ of rank $3$ over $F$ such that 
for all $\fp\in S$ the space $(\phi\otimes\mathbb{F}_\mathfrak{p})[\mathfrak{l}]$ contains a nontrivial $\mathbb{F}_\mathfrak{p}$-rational point, 
but $\phi[\fl]$ is irreducible.

In this subsection, we set $A=\mathbb{F}_q[T]$, $q=p^e$ a prime power with $p\geqslant 5$ such that $x^2+x+1$ is irreducible in $\mathbb{F}_q[x]$, and $F=\mathbb{F}_q(T)$. Consider the Drinfeld module $\varphi$ of rank $2$ defined by
$$\varphi_T=T+T\tau+T^q\tau^2.$$

\begin{lem}\label{image sl2}
The mod $(T)$ representation ${\bar{\rho}}_{\varphi,T}:G_F\rightarrow {\rm{GL}}_2(\mathbb{F}_q)$ has image equal to ${\rm{SL}}_2(\mathbb{F}_q)$. 

\end{lem}
\begin{proof}
By \cite{Hei03} Proposition 4.7.1, we have $\det\circ\bar{\rho}_{\varphi,T}=\bar{\rho}_{\psi,T}$ where $\psi$ is the Drinfeld module of rank $1$ defined by $\psi_T=T-T^q\tau$. Thus the representation $\bar{\rho}_{\psi,T}$ is trivial, which implies the image of $\bar{\rho}_{\varphi,T}$ lies in ${\rm{SL}}_2(\mathbb{F}_q)$.

Next, we prove the image of $\bar{\rho}_{\varphi,T}$ contains a subgroup of order $q$. Consider the decomposition subgroup ${\rm{Gal}}(F^{sep}_{(T)}/F_{(T)})$ of $G_F$. Since $F_{(T)}(\varphi[T])$ is the smallest extension of $F_{(T)}$ such that ${\rm{Gal}}(F^{sep}_{(T)}/F_{(T)}(\varphi[T]))$ acts trivially on $\varphi[\mathfrak{l}]$, we have $\bar{\rho}_{\varphi,T}({\rm{Gal}}(F^{sep}_{(T)}/F_{(T)})) \cong {\rm{Gal}}(F_{(T)}(\varphi[T])/F_{(T)}).$ By looking at the Newton's polygon of $\varphi_T(x)/x=T+Tx^{q-1}+T^qx^{q^2-1}$, we know $\varphi_T(x)$ has roots with valuation equal to $-\frac{1}{q}$. Therefore, $F_{(T)}(\varphi[T])$ must contain a subfield $M$ which is Galois over $F_{(T)}$ and its ramification index  $e[M:F_{(T)}]$ is divisible by $q$. Thus the order of $\bar{\rho}_{\varphi,T}(G_F)$ is divisible by $q$.

Finally, we prove the Galois module $\varphi[T]$ is irreducible. Let $\mathfrak{p}=(T-c)$ be a degree $1$ prime ideal of $A$ with $c\in\mathbb{F}_q^*$. By \cite{Chen20} Proposition $7$ and Proposition $8$, we may compute the characteristic polynomial $P_{\varphi,(T-c)}(x)=-c^{-1}(T-c)+ax+x^2\in A[x]$ and $a$ belongs to $\mathbb{F}_q$.  Because $P_{\varphi,\mathfrak{p}}(x)$ is also the characteristic polynomial of Frobenius endomorphism of $\varphi\otimes\mathbb{F}_{\mathfrak{p}}$ acting on $T_{\mathfrak{l}}(\varphi\otimes\mathbb{F}_{\mathfrak{p}})$, we have
$$-c^{-1}(\varphi\otimes\mathbb{F}_{\mathfrak{p}})_{T-c}+(\varphi\otimes\mathbb{F}_\mathfrak{p})_a\tau+\tau^2=0$$
As $\varphi_T=T+T\tau+T^q\tau^2$, we have $(\varphi\otimes \mathbb{F}_\mathfrak{p})_{T-c}=c\tau+c\tau^2$. Hence the above equation implies $a=1$. Therefore, the characteristic polynomial of $\bar{\rho}_{\varphi,T}({\rm{Frob}}_\mathfrak{p})$ is equal to $x^2+x+1$. By our choice of $q$, we can deduce that $\phi[T]$ is irreducible.
By \cite{Zy11} Lemma A.1, we have $\bar{\rho}_{\varphi,T}(G_F)\supseteq {\rm{SL}}_2(\mathbb{F}_q)$. Thus the proof is now complete.
\end{proof}

Now we consider the representation $\rho$ of $G_F$ defined by the following composition:
$$\rho:G_F\xrightarrow{{\bar{\rho}}_{\varphi,T}}{\rm{SL}}_2(\mathbb{F}_q)\xrightarrow{{\rm{projection}}} {\rm{PSL}}_2(\mathbb{F}_q)\xrightarrow[{\rm{exceptional\ isomorphism}}]{\sim}\Omega_3(\mathbb{F}_q)\subset {\rm{SO}}_3(\mathbb{F}_q)\subset {\rm{GL}}_3(\mathbb{F}_q).$$
Here $\Omega_3(\mathbb{F}_q)$ is the subgroup of ${\rm{SO}}_3(\mathbb{F}_q)$ of index $2$ generated by $\left(\begin{array}{ccc}1 & 2 & -1 \\-1 & -1 & 0 \\-1 & 0 & 0\end{array}\right)\ {\rm{and}}\ \left(\begin{array}{ccc}\xi^{-2} & 0 & 0 \\0 & 1 & 0 \\0 & 0 & \xi^2\end{array}\right)$, where 
$\xi$ is a generator of $\mathbb{F}_q^*$; see \cite{RyTa98}, section 4.6. 
One can also refer to page 53 in \cite{Gr02} for the formal definition of the group $\Omega_3$ defined over a field.
By the Boston-Ose theorem (\cite{Bose00} Theorem 6.1), $\rho$ is arised from the mod $(T)$ Galois representation associated to some Drinfeld $A$-module $\phi$ over $F$ of rank $3$. In other words, the mod $(T)$ representation ${\bar{\rho}}_{\phi,T}$ associated to $\phi$ has image equal to $\Omega_3(\mathbb{F}_q)$.

\begin{lem}
For almost all prime ideal $\mathfrak{p}$ of $A$, the $T$-torsion $(\phi\otimes\mathbb{F}_\mathfrak{p})[T]$ contains a nontrivial $\mathbb{F}_\mathfrak{p}$-rational point.
\end{lem}

\begin{proof}
Consider those prime ideals $\mathfrak{p}\neq(T)$ where $\phi$ has good reduction at $\mathfrak{p}$. By Proposition \ref{red}, it suffices to prove that the characteristic polynomial $P_{\phi,\mathfrak{p}}(x)$ of ${\rm{Frob}}_\mathfrak{p}$ acting on $T_{(T)}(\phi)$ satisfies 
$$P_{\phi,\mathfrak{p}}(1)\equiv 0 \mod T.$$
In other words, we want to prove that the characteristic polynomial ${\bar{P}}_{\phi,\mathfrak{p}}(x)$ of ${\rm{Frob}}_\mathfrak{p}$ acting on $\phi[T]$ has a linear factor $(x-1)$. Since the image of the mod $(T)$ representation $\bar{\rho}_{\phi,T}$ is a subgroup of ${\rm{SO}}_3(\mathbb{F}_q)$, the proof is complete whenever every nontrivial element of ${\rm{SO}}_3(\mathbb{F}_q)$ fixes a point in $\mathbb{F}_q^3$.
This has been proved in \cite{Gr02}, Corollary 6.10.

\end{proof}

\begin{prop}
$\Omega_3(\mathbb{F}_q)$ acts on $\mathbb{F}_q^3 $ irreducibly. 
\end{prop}
\begin{proof}

From \cite{RyTa98}, section 4.6, there is a basis of $\mathbb{F}_q^3$ such that the generators of $\Omega_3(\mathbb{F}_q)$ are matrices
$$nx=\left(\begin{array}{ccc}1 & 2 & -1 \\-1 & -1 & 0 \\-1 & 0 & 0\end{array}\right)\ {\rm{and}}\ h=\left(\begin{array}{ccc}\xi^{-2} & 0 & 0 \\0 & 1 & 0 \\0 & 0 & \xi^2\end{array}\right)$$
with respect to that basis. 
Suppose that there is a proper nontrivial subspace $V$ of $\mathbb{F}_q^3$ which is fixed under the action of $\Omega_3(\mathbb{F}_q)$. Such $V$ cannot be of dimension $1$ by computing the eigenvectors of $nx$ and $h$. Thus $V$ must be a $2$-dimensional space. Write
$$V={\rm{span}}\left\{\left(\begin{array}{c}a \\b \\c\end{array}\right),\ \left(\begin{array}{c}x \\y \\z\end{array}\right)\right\}.$$
We compute
$$
\begin{array}{ccll}

&nx\cdot\left(\begin{array}{c}a \\b \\c\end{array}\right)&=&\left(\begin{array}{c}2a+2b-c \\-a \\0\end{array}\right)\in V\\
\ &\ &\ &\ \\
\Rightarrow&nx\cdot\left(\begin{array}{c}2a+2b-c \\-a \\0\end{array}\right)&=&\left(\begin{array}{c}-(-2b+c) \\-2b+c \\-2b+c\end{array}\right)+\left(\begin{array}{c}0 \\-a \\-2a\end{array}\right)
\end{array}
$$
\begin{claim}
$a\neq 0$.
\end{claim}
\begin{proof}[Proof of claim]  There are two cases to consider:  
\begin{itemize}
\item[Case (i).] If $a=0$ and $2b-c\neq 0$, then we have 
$$\left(\begin{array}{c}-1 \\1 \\1\end{array}\right)\in V\ {\rm{and}}\ \left(\begin{array}{c}1 \\0 \\0\end{array}\right)\in V.$$
Hence we have a basis of $V$. 

However, $h\cdot\left(\begin{array}{c}-1 \\1 \\1\end{array}\right)=\left(\begin{array}{c}-\xi^2 \\1 \\\xi^2\end{array}\right)$ does not belong to $V$, which gives a contradiction since $V$ is fixed under $\Omega_3(\mathbb{F}_q)$-action.

\item[Case (ii).] If $a=2b-c=0$, then the base vector $\left(\begin{array}{c}a \\b \\c\end{array}\right)$ gives
$\left(\begin{array}{c}0 \\1 \\2\end{array}\right)\in V.$
Moreover, we have
$$nx\cdot \left(\begin{array}{c}0 \\1 \\2\end{array}\right)=\left(\begin{array}{c}0 \\-1 \\0\end{array}\right)\in V.$$
Hence we have found a basis for $V$.

However, $nx\cdot h\cdot\left(\begin{array}{c}0 \\1 \\2\end{array}\right)=\left(\begin{array}{c}2-2\xi^2 \\-1 \\0\end{array}\right)$ does not belong to $V$ because $\xi^2\neq 1$. This gives a contradiction.
\end{itemize}
\end{proof}

Thus, the entry $a$ in $\left(\begin{array}{c}a \\b \\c\end{array}\right)$ must not equal to $0$. Similarly, the entry $x$ in $\left(\begin{array}{c}x \\y \\z\end{array}\right)$ is not equal to $0$ as well. Hence there are some nonzero $\alpha,\ \beta\in\mathbb{F}_q$ such that
$$\alpha\left(\begin{array}{c}a \\b \\c\end{array}\right)+\beta\left(\begin{array}{c}x \\y \\z\end{array}\right)=\left(\begin{array}{c}0 \\s \\t\end{array}\right)\in V.$$
We may write
$$V={\rm{span}}\left\{\left(\begin{array}{c}0 \\s \\t\end{array}\right),\ \left(\begin{array}{c}x \\y \\z\end{array}\right)\right\}.$$
This contradicts the claim proved above. 
Hence there is no such nontrivial proper invariant subspace $V$.
\end{proof}

Unfortunately, the proof of the Boston-Ose theorem only implies the existence of $\phi$ without providing a method 
for writing down an equation for $\phi_T$. It seems like an interesting problem to write down $\phi_T$ 
such that $\bar{\rho}_{\phi, T}(G_F)\cong \Omega_3(\F_q)$. 


\section*{Acknowledgements}
The author would like to thank his advisor Professor Mihran Papikian for helpful comments and suggestions on carrying out this paper.

\bibliographystyle{alpha}
\bibliography{On_torsion_points_for_DM}

\begin{thebibliography}{{Zyw}11}

\bibitem[BO00]{Bose00}
Nigel Boston and David~T. Ose.
\newblock Characteristic {$p$} {G}alois representations that arise from
  {D}rinfeld modules.
\newblock {\em Canad. Math. Bull.}, 43(3):282--293, 2000.

\bibitem[Che20]{Chen20}
{Surjectivity of the adelic Galois representation associated to a Drinfeld
  module of rank 3}.
\newblock {\em J. Number Theory}, to appear, 2020.

\bibitem[CL19]{ChLe19}
Imin Chen and Yoonjin Lee.
\newblock Explicit surjectivity results for {D}rinfeld modules of rank 2.
\newblock {\em Nagoya Math. J.}, 234:17--45, 2019.

\bibitem[CR06]{CuRe62}
Charles~W. Curtis and Irving Reiner.
\newblock {\em Representation theory of finite groups and associative
  algebras}.
\newblock AMS Chelsea Publishing, Providence, RI, 2006.
\newblock Reprint of the 1962 original.

\bibitem[Cul19]{Cull19}
John Cullinan.
\newblock Fixed-point subgroups of {${\rm GL}_3(q)$}.
\newblock {\em J. Group Theory}, 22(5):893--914, 2019.

\bibitem[Gek91]{Gek91}
Ernst-Ulrich Gekeler.
\newblock On finite {D}rinfeld modules.
\newblock {\em J. Algebra}, 141(1):187--203, 1991.

\bibitem[Gos96]{Gos96}
David Goss.
\newblock {\em Basic structures of function field arithmetic}, volume~35 of
  {\em Ergebnisse der Mathematik und ihrer Grenzgebiete (3) [Results in
  Mathematics and Related Areas (3)]}.
\newblock Springer-Verlag, Berlin, 1996.

\bibitem[Gro02]{Gr02}
Larry~C. Grove.
\newblock {\em Classical groups and geometric algebra}, volume~39 of {\em
  Graduate Studies in Mathematics}.
\newblock American Mathematical Society, Providence, RI, 2002.

\bibitem[Kat81]{Kat81}
Nicholas~M. Katz.
\newblock Galois properties of torsion points on abelian varieties.
\newblock {\em Invent. Math.}, 62(3):481--502, 1981.

\bibitem[RT98]{RyTa98}
L.~J. Rylands and D.~E. Taylor.
\newblock Matrix generators for the orthogonal groups.
\newblock {\em J. Symbolic Comput.}, 25(3):351--360, 1998.

\bibitem[{van}03]{Hei03}
Gerrit {van der Heiden}.
\newblock {\em Weil pairing and the Drinfeld modular curve}.
\newblock PhD thesis, University of Groningen, 2003.
\newblock Relation: $https://www.rug.nl/ date_submitted:2003$ Rights:
  University of Groningen.

\bibitem[{Zyw}11]{Zy11}
David {Zywina}.
\newblock {Drinfeld modules with maximal Galois action on their torsion
  points}.
\newblock {\em arXiv e-prints}, page arXiv:1110.4365, October 2011.

\end{thebibliography}
\end{document}